\documentclass{article}
\usepackage{amsmath,amssymb,amsthm}
\usepackage{amscd}

\newtheorem{theorem}{Theorem}[section]
\newtheorem{lemma}[theorem]{Lemma}

\newtheorem{proposition}[theorem]{Proposition}
\newtheorem{corollary}[theorem]{Corollary}
\newtheorem{definition}[theorem]{Definition}

\newtheorem{example}[theorem]{Example}
\newtheorem{problem}[theorem]{Problem}
\newtheorem{remark}[theorem]{Remark}
\newtheorem{assumption}[theorem]{Assumption}

\newcommand{\Z}{{\mathbb Z}}
\newcommand{\Q}{{\mathbb Q}}
\newcommand{\Zp}{{\mathbb Z}_p}
\newcommand{\Qp}{{\mathbb Q}_p}

\newcommand{\va}{{\bf a}}
\newcommand{\vb }{{\bf b}}
\newcommand{\vc}{{\bf c}}
\newcommand{\ve}{{\bf e}}
\newcommand{\vx}{{\bf x}}
\newcommand{\vy}{{\bf y}}

\newcommand{\vzero}{{\bf 0}}

\newcommand{\diag}{{\rm diag}}

\newcommand{\bs}{\backslash}

\newcommand{\Aut}{{\rm Aut}}
\newcommand{\End}{{\rm End}}

\newcommand{\abgr}{{\mathcal M}}

\newcommand{\mo}{\Delta}
\newcommand{\tmo}{\tilde{\Delta}}
\newcommand{\gr}{\Gamma}
\newcommand{\grsupa}{\gr_A} 
\newcommand{\grsuba}{\gr^A}
\newcommand{\tgr}{\tilde{\Gamma}}

\newcommand{\tR}{\tilde{R}}
\newcommand{\tP}{\tilde{P}}

\newcommand{\lp}{L_p}
\newcommand{\Hn}{\mathcal{H}^{(n)}}
\newcommand{\Ho}{\mathcal{H}^{(1)}}

\newcommand{\rlp}{R_{L_p}}
\newcommand{\plp}{P_{L_p}}

\newcommand{\rHnp}{R_{\mathcal{H}^{(n)}_p}}
\newcommand{\pHnp}{P_{\mathcal{H}^{(n)}_p}}
\newcommand{\pHop}{P_{\mathcal{H}^{(1)}_p}}
\newcommand{\pHtp}{P_{\mathcal{H}^{(2)}_p}}

\newcommand{\glp}{G_{L_p}}
\newcommand{\dlp}{\Delta_{L_p}}
\newcommand{\gamlp}{\Gamma_{L_p}}

\newcommand{\tlp}{T_{L_p}}

\newcounter{assumpcnt}
\begin{document}

\markboth{F.~Hyodo}{A noncommutative Hecke ring and the Hecke series for $GSp_4$}

%

%

\title{A formal power series over a noncommutative Hecke ring and the rationality of the Hecke series for $GSp_4$
}

\author{Fumitake Hyodo
}


%

\date{}

\maketitle


\begin{abstract}
The present paper studies Hecke rings derived by the automorphism groups of certain algebras $\lp$ over the ring of $p$-adic integers.
Our previous work considered the case where $\lp$ is the Heisenberg Lie algebra (of dimension 3) over the ring of $p$-adic integers. Although this Hecke ring is noncommutative, we showed that a formal power series with coefficients in this Hecke ring satisfies an identity similar to the rationality of the Hecke series for $GL_2$ due to E.~Hecke.
In the present paper, we establish an analogous result in the case of the Heisenberg Lie algebra of dimension 5 over the ring of $p$-adic integers. In this case, our identity is similar to the rationality of the Hecke series for $GSp_4$, due to G.~Shimura.
\end{abstract}



%

\section{Introduction}	
{
The present paper is a continuation of our previous works \cite{Hy}, \cite{Hy2} and \cite{Hy4}.
We discuss identities for formal power series
with coefficients in certain noncommutative Hecke rings.
Our study is related to the work due to \cite[Hecke]{He}, \cite[Shimura]{S1} and \cite[Andorianov]{A1}. 
Let $p$ be a fixed prime number. They showed that, for each positive integer $n$, the Hecke series $P_n(X)$ for the general symplectic group $GSp_{2n}(\Qp)$ of genus $n$ over the $p$-adic field $\Qp$ has rationality as follows:

\begin{theorem}[Rationality Theorem, \cite{A1}, \cite{He}, \cite{S1}]
\label{thm:Rationality Theorem}
Denote by $R_n$ the Hecke ring associated with $GSp_{2n}(\Qp)$.
Then there exist elements $q^{(n)}_1,...,q^{(n)}_{2^n}$ of $R_n$, and a polynomial $g^{(n)}(X) \in R_n[X]$ such that
\begin{equation}\label{HAeq}
\sum_{i=0}^{2^n} q^{(n)}_i X^i P_n(X)  = g^{(n)}(X),
\end{equation}
where $q^{(n)}_{0} = 1$.
\end{theorem}

The above theorem is established by \cite{He} for the case $n=1$,
by \cite{S1} for the case $n=2$, and by \cite{A1} for any $n\geq 3$.
Furthermore, it is known that $R_n$ is isomorphic to the (commutative) polynomial ring of $(n+1)$ variables. Thus $\{q_i^{(n)}\}_i$ and $g^{(n)}(X)$ are unique.

Throughout the present paper, algebra implies an abelian group with a bi-additive product (e.g., an associative algebra, a Lie algebra). Let $L$ be an algebra that is free of finite rank as an abelian group, and let $\lp = L \otimes \Z_p$. In \cite{Hy2}, the Hecke rings $\rlp$ derived by the automorphism groups of $\lp$ were introduced, and in \cite{Hy4}, the formal power series $\plp(X)$ with coefficients in the Hecke rings were defined.
We call $\plp(X)$ and $\rlp$ the local Hecke series and the local Hecke ring associated with $L$, respectively.
%

In the present paper, we focus on the case where $L$ is the Heisenberg Lie algebra $\Hn$ over $\Z$ of dimension $2n+1$. Here, the ring homomorphisms $s_n:R_n \to \rHnp$, $\phi_n: \rHnp \to R_n$ and $\theta_n: \rHnp \to  \rHnp$ satisfying the following property are constructed. They are presented in Sections \ref{section:BasicTheory} and \ref{section:Heisenberg Lie algebras}.

\begin{proposition}[Proposition \ref{prop:Relation:Heisenberg}, 
Corollary \ref{corollary:phiDtD}]
\label{prop:mainprop}
The ring homomorphisms $\phi_n$, $s_n$ and $\theta_n$ satisfy the following properties:
\begin{enumerate}
\setlength{\itemsep}{5pt}
\item
$\phi_n \circ s_n = id_{R_n}$. 
\item
$\phi_n \circ \theta_n = \phi_n$.
\item
$\pHnp^{\phi_n}(X) = P_n(p^{n}X^{n+1})$.
\end{enumerate}
\end{proposition}

We shall drop the subscript $n$ from $\phi_n$, $s_n$ and $\theta_n$ when no confusion can occur. 

Our previous work \cite{Hy} shows the following identity as well as the noncommutativity of the coefficients of the local Hecke series associated with the Heisenberg Lie algebra $\Ho$ of dimension $3$.

\begin{theorem}[{\cite[Theorem 7.8]{Hy}}]
We set $P(X) = \pHop(X)$, $g(X)=g^{(1)}(X)$ and $q_i = q^{(1)}_i$ for $i=0,1,2$.
Then we have the following identity.
\begin{equation}\label{eq:previous}
P^{\theta^2}(X) + q^s_1 Y P^{\theta}(X) + q^s_2 Y^2 P(X) = g^s(Y),
\end{equation}
where $Y = pX^{2}$. 
\end{theorem}

Our main purpose of the present paper is to extend this result. Precisely,
we establish the noncommutativity of the coefficients of $\pHtp(X)$
, moreover, the following identity.

\begin{theorem}[Theorem \ref{thm:main}]
We set  $P(X) = \pHtp(X)$, $g(X)=g^{(2)}(X)$, and $q_i = q^{(2)}_i$ for $0\leq i\leq4$.
Then we have the following identity.
\begin{equation}\label{eq:present}
\sum^{4}_{i=0} q^s_i Y^i P^{\theta^{4-i}}(X) = g^s(Y),
\end{equation}
where $Y = p^2 X^{3}$. 
\end{theorem}

It should be noted that
Equalities (\ref{eq:previous}) and (\ref{eq:present}) recover the rationality theorems of Hecke and Shimura 
via the morphisms $\phi_1$ and $\phi_2$, respectively. This is an immediate consequence of Proposition \ref{prop:mainprop}. 

For the case $n \geq 3$, we show the noncommutativity of the coefficients of $\pHnp(X)$ (cf. Theorem \ref{thm:noncommutativity}), and raise the following unsolved problem.%

\begin{problem}[Problem \ref{myprob}]
We set $P(X)=\pHnp(X)$, $g(X)=g^{(n)}(X)$, and $q_i = q^{(n)}_i$ for $0\leq i \leq 2^n$.
Does $P(X)$ satisfy the following identity ?
\begin{equation}\label{mainprobeq}
\sum^{2^n}_{i=0} q^s_i Y^i P^{\theta^{N-i}}(X) = g^s(Y).
\end{equation}
Where $Y = p^nX^{n+1}$ and $N=2^n$. 
\end{problem}

The outline of the present paper is as follows.
In Section \ref{section:BasicTheory}, we recall the definition of Hecke rings and study the morphisms $s$, $\phi$ and $\theta$.
Section \ref{section:RingsOfHeckeAlgebras} describes the local Hecke rings and the local Hecke series associated with algebras.
In Section \ref{section:Heisenberg Lie algebras}, we focus on the case of the Heisenberg Lie algebras, and then show the noncommutativity of the coefficients of $\pHnp(X)$ for all $n$. 
Finally, our main theorem is proved in Section \ref{section:MainResult}.
}

\section{Abstract Hecke rings and some morphisms}\label{section:BasicTheory}
{
\newcommand{\tT}{\tilde{T}}
\newcommand{\set}[1]{\left\{ #1 \right\}}
\newcommand{\card}[1]
{
	\left| #1 \right|
}

\newcommand{\thetaCzero}{C_0}
%
In this section, we recall the definition of Hecke rings and define the morphisms $s$, $\theta$, and $\phi$ which are used in section \ref{section:Heisenberg Lie algebras}.

First, we recall the definition of Hecke rings.
For more details, refer to \cite[Shimura, Chapter 3]{S2}.
Let $G$ be a group, $\mo$ be a submonoid of $G$, and $\gr$ be a subgroup of $\mo$. We assume that 
the pair $(\gr, \mo)$ is a double finite pair, i.e., for all $A \in \mo$, 
$\gr \bs \gr A \gr$ and $\gr A \gr / \gr$ are finite sets.
Then, one can define the Hecke ring $R = R(\gr, \mo)$ associated with the pair $(\gr, \mo)$ as follows:
\begin{itemize}
\item The underlying abelian group is the free abelian group on the set $\gr \bs \mo /\gr$.
\item The product of $(\gr A \gr)$ and $(\gr B \gr)$ is defined by
	\[\sum_{\gr C \gr \in \gr \bs \mo /\gr} m_C (\gr C \gr), \]
where 
\begin{eqnarray*}
m_C &=& 
\card{ 
\set{
\gr \beta \in \gr \bs \gr B \gr \ |\ C\beta^{-1} \in \gr A\gr 
}
}\\
 &=& 
\card{ 
\set{
\alpha \gr \in \gr A \gr / \gr \ |\ \alpha^{-1}C \in \gr B\gr 
}
}
.
\end{eqnarray*}
\end{itemize}
Note that $m_C \not =  0$ if and only if $C \in \gr A \gr B \gr$.

Define the element $T_{\gr, \mo}(A)$ of $R(\gr, \mo)$ by $\gr A \gr$ for every $A \in \mo$. 
We also define $\deg_{\gr, \mo}(\gr A\gr)$ by $\left| \gr \bs \gr A \gr\right|$ for every $A \in \mo$.
The map $\deg_{\gr, \mo}$ extends by linearity to a homomorphism from $R(\gr, \mo)$ to $\Z$, and
forms a ring homomorphism.
We will write simply ``$T(A)$", ``$\deg$" and ``$\deg T(A)$", for
$T_{\gr, \mo}(A)$, $\deg_{\gr, \mo}$ and $\deg_{\gr, \mo} ( T_{\gr, \mo}(A) )$
respectively when no confusion can arise.
\smallskip

Fix a double finite pair $(\gr, \mo)$ and a both sides $\mo$-module $\abgr$. 
Next, we construct a Hecke ring $\tR$ using $\mo$, $\gr$ and $\abgr$.
We define the monoid $\tmo= \tmo({\mo,\abgr})$ as follows.
\begin{itemize}
\item The underlying set is $\mo \times \abgr$.
\item The operation is defined by 
	$(A, \va) \cdot (B, \vb) = (AB, A \vb + \va B)$ for $A,B \in \mo$,  $\va,\ \vb \in \abgr$.
\end{itemize}
We also define $\tgr = \tgr({\gr,\abgr})$ by the subgroup of $\tmo$ whose underlying set is $\gr \times \abgr$.
Notice that the identity element of $\tmo$ is $(E,\vzero)$, 
where $E$ and $\vzero$ are the identity elements of $\mo$ and $\abgr$ respectively.
It is easy to see $(X, \vx)^{-1} = (X^{-1}, - X^{-1}\vx X^{-1})$ for each $(X,\vx) \in \tgr$.

From now on, we assume the following:
\begin{assumption}\ \label{assumption:deg}
\begin{enumerate}
\item
$\abgr /(A\abgr \cap \abgr A) $ is a finite set for every $A \in \mo$.
\item \setcounter{assumpcnt}{\arabic{enumi}}
$\abgr$ is a sub-$\mo$-module of a both sides $G$-module $\abgr'$. 
\end{enumerate}
\end{assumption}
\medskip

Then, the monoid $\tmo$ is naturally a submonoid of the group $\tmo(G,\abgr')$.
Let us calculate the degree of each $(A,\va) \in \tmo$ and prove the double finiteness of the pair $(\tgr,\tmo)$.
We put $\grsupa = A^{-1} \gr A \cap \gr$ and $\grsuba = \gr \cap A \gr A^{-1}$.
Note that there are natural bijections of
$\grsupa \bs \gr$ onto $\gr \bs \gr A \gr$ and $\gr / \grsuba$ onto $\gr A \gr / \gr$.


\begin{proposition}\label{BasicTheory:DoubleCoset}
For any two elements $(A,\va)$ and $(A,\vb)$ of $\tmo$, $\tgr (A,\va) \tgr = \tgr (A,\vb) \tgr$ 
if and only if there exists $X \in \grsuba$ such that 
\[
	 X  \ast \va  \equiv \vb \mod A  \abgr +\abgr A  .
\]
Where $X  \ast \va = X \va A  ^{-1}X^{-1}  A $.

\end{proposition}

\begin{proof}
$\tgr (A,\va) \tgr = \tgr (A,\vb) \tgr$ if and only if there exist $(X ,\vx ), (Y ,\vy ) \in \tgr$ such that
$(X ,\vx )(A ,\va) = (A ,\vb )(Y ,\vy )$, which is equivalent to
\[Y   = A  ^{-1}X  A  , \quad X \ast \va - \vb  = A  \vy  Y  ^{-1} - \vx  X ^{-1}A.\]
This completes the proof.
\end{proof}

%
\begin{proposition}
For $(A,\va) \in \tmo$, we have
\[\card{\tgr \bs \tgr (A,\va) \tgr} = [ \gr : \gr_A][ A\abgr :A \abgr \cap \abgr A] 
\left|  \gr^A* (\va \mod  A \abgr + \abgr A)\right|.\]
 where $ \grsuba  \ast (\va  \mod A\abgr )$ means the orbit of 
 $(\va  \mod A\abgr + \abgr A)$ in $\abgr / (A \abgr + \abgr A)$ under $ \grsuba $ acting by $\ast$.
Particularly, $\card{\tgr \bs \tgr (A,\va) \tgr}$ is finite.
\end{proposition}

\begin{proof}
We put \[\tgr_{(A,\va)} = 
\set{\left. (X,\vx) \in \tgr \ \right| \ (A,\va) (X,\vx) = (Y, \vy) (A,\va)
 \ \text{for some $(Y,\vy) \in \tgr$}}.\]
Then $\card{\tgr \bs \tgr (A,\va) \tgr} = \card{\tgr_{(A,\va)} \bs \tgr}$.
$(X,\vx) \in \tgr_{(A,\va)}$ if and only if $(X,\vx)$ has the following three conditions:
\begin{eqnarray}
& & X \in  \gr_A, \label{gr1_1}\\
& & (AXA^{-1}) \ast \va   - \va   \in A \abgr  +  \abgr A, \label{gr1_2}\\
& & A\vx X^{-1} \in  (AXA^{-1}) \ast \va  - \va  + \abgr A.
\end{eqnarray}
Let $\gr_1$ be the image of $\tgr_{(A,\va)}$ by the canonical projection $\tgr \to \gr$, 
and put ${\cal N} = \set{\vx \in \abgr \ | \ A\vx \in A \abgr \cap \abgr A}$. 
Then we have a commutative diagram as follows.
\[
	\begin{CD}
	1 @> >> \abgr @> >> \tgr @> >> \gr @> >> 1 \\
	  @.      @AA A   @AA A    @AA A   \\
	1 @> >> {\cal N} @> >> \tgr_{(A,\va)} @> >> \gr_1@> >> 1 \\
	\\
	\end{CD}
\]
where $\abgr \to \tgr$ is the canonical embedding, and each of the three vertical morphisms is the inclusion. 
Since the two horizontal sequences are exact, we have  

\[
[\tgr:\tgr_{(A,\va)}]
=
[\gr : \grsupa][\grsupa:\gr_1][\abgr : {\cal N}].
\]
Clearly, $\abgr / {\cal N}$ is naturally isomorphic to $A\abgr / A \abgr  \cap  \abgr A$. Since $\gr_1$ coinsides with the subgroup of $\gr$ consisting of elements satisfying (\ref{gr1_1}) and (\ref{gr1_2}), 
$A \gr _1 A^{-1}$ is the stabilizer subgroup of 
$\left(\va \mod  A \abgr + \abgr A\right)$
in $\gr^A$.
Thus we have the desired identity.
\end{proof}
A similar argument shows that $\card{\tgr (A,\va) \tgr / \tgr}$ is also finite. Thus we have:

\begin{corollary}\label{cor:DoubleFiniteness}
The pair $(\tgr,\tmo)$ is double finite. 
\end{corollary}

We write $\tR = R(\tgr,\tmo)$ and $\tT = T_{\tgr,\tmo}$ .
The above proposition is equivalent to the following statement: 

\begin{corollary}\label{cor:BasicTheory:Degree:Degree}
For $(A,\va) \in \tmo$, we have
\[\deg\tT(A,\va) =[ \gr : \gr_{A}][ A\abgr :A \abgr \cap \abgr A] 
\left|  \gr^A* (\va \mod  A \abgr + \abgr A)\right|.\]
\end{corollary}
\smallskip

The following three subsections define the morphisms $s$, $\theta$ and $\phi$, respectively.
%
From now on, we assume the following:
\begin{assumption}\label{assumption:s}
$A \abgr \supset \abgr A$ for every $A \in \mo$.
\end{assumption}
\subsection{The injection $s$}
%
We define $\psi:  \tmo \to \mo $ and $s_0 :  \mo \to \tmo $ by ($(C,\vc) \mapsto C)$ and $(C \mapsto (C,\vzero))$, respectively.

\begin{lemma}\label{mult1to11}
	For $A,B \in\mo$, Put 
	\[\tilde{\mathcal X}= \tgr \backslash \tgr (A,\vzero)\tgr (B,\vzero)  \tgr/ \tgr
	,\text{and}\ 
	{\mathcal X} = \gr \backslash \gr A\gr B  \gr/ \gr.\]
Let $\bar{\psi}:\tilde{\mathcal X}
\to 
{\mathcal X}$
and
$\bar{s_0}: {\mathcal X}
\to \tilde{\mathcal X}$
be the maps defined by $\psi$ and $s_0$, respectively. Then
$\bar{\psi}$ and $\bar{s_0}$ are the inverse to each other.
\end{lemma}

\begin{proof}
It is clear that $\bar{\psi} \circ \bar{s_0}$ is the identity.
We shall show $\bar{s_0}\circ \bar{\psi} $ is also the identity.
For $(X,\vx) \in \tgr$,  $(A,\vzero) (X,\vx)  (B,\vzero) =(AXB, A\vx B)$.
Since $A\abgr B  = AX\abgr B \subset AXB\cdot \abgr$,
by Proposition \ref{BasicTheory:DoubleCoset}, we have
\[\tgr  (A,\vzero) (X,\vx)  (B,\vzero) \tgr = \tgr(AXB, \vzero)\tgr.\]
This completes the proof.
\end{proof}

\begin{lemma}\label{mult1to12}
	For $A,B,C \in\mo$, let $\alpha=(A,\vzero),\beta=(B,\vzero), \gamma=(C,\vzero)$. We put 
 \[\tilde{\mathcal X}=\{\tgr\delta \in \tgr \backslash \tgr  \beta \tgr;\gamma\delta^{-1} \in  \tgr \alpha \tgr \} \]
 and
 \[{\mathcal X}=
 \{\gr D \in \gr \backslash \gr  B  \gr;
			CD^{-1} \in  \gr A \gr \}
 .\]
Let $\bar{\psi}:\tilde{\mathcal X}
\to 
{\mathcal X}$
and
$\bar{s_0}: {\mathcal X}
\to \tilde{\mathcal X}$
be the maps defined by $\psi$ and $s_0$, respectively. Then
$\bar{\psi}$ and $\bar{s_0}$ are the inverse to each other, especially,
$|\tilde{\mathcal X}| = |{\mathcal X}|$.
\end{lemma}

\begin{proof}
It is essential to prove $\bar{s_0}\circ \bar{\psi}$ is the identity.
For $(X,\vx) \in \tgr$, we put
$\delta =\beta (X,\vx)=(BX,B\vx)$,
and assume 
$\gamma\delta^{-1} \in  \tgr \alpha \tgr$.
It suffices to prove \[\tgr\delta =\tgr (BX,\vzero).\]
Since
$\gamma\delta^{-1} = (CX^{-1}B^{-1},-CX^{-1}\vx(BX)^{-1})$,
we have 
$CX^{-1}B^{-1} \in \gr A \gr$.
Hence $(CX^{-1}B^{-1},\vzero) \in \tgr \alpha \tgr$.
Thus
\[(CX^{-1}B^{-1},-CX^{-1}\vx(BX)^{-1}) \in \tgr(CX^{-1}B^{-1},\vzero)\tgr.\]
By Proposition \ref{BasicTheory:DoubleCoset},
$-CX^{-1}\vx(BX)^{-1} \in CX^{-1}B^{-1}\abgr$,
namely, $\vx \in B^{-1}\abgr BX$.
Put $\vx = B^{-1}\vy BX$, then
\[\delta = (BX,\vy BX) = (E,\vy)(BX,\vzero).\]
This completes the proof.
\end{proof}

By the above two lemmas, we have:

\begin{proposition}\label{section}
	The natural injection  $ s: R \to \tR $ defined by
	\[T(A) \mapsto \tT(A,\vzero), \quad for \ A \in \mo, \]	
	is a ring homomorphism .
\end{proposition}
\bigskip

\subsection{The endomorphism $\theta$}
\label{section:TheEndomorphismTheta}
\newcommand{\ecent}
{
C_0
}
%
Let $\ecent$ be an element of the center of $\mo$.
First, we calculate the products $\tilde{T}(\ecent,\vzero)\tilde{T}( A ,\va )$ and $\tilde{T}( A ,\va )\tilde{T}(\ecent,\vzero)$ for each $(A,\va) \in \tmo$.

\begin{lemma}\label{BasicTheory:RightAction}
For all $(A,\va) \in \tmo$,
	\[\tilde{T}(\ecent,\vzero)  \tilde{T}( A ,\va )  = \tilde{T}(\ecent  A , \ecent \va ).\]
\end{lemma}

\begin{proof}
{
\newcommand{\tmp}
{
(\ecent  ,\vzero)	
}
\newcommand{\tmpb}
{
(X,\vx)
}
\newcommand{\tmpc}
{
(A,\va) 
}
\newcommand{\tmpd}
{
(\ecent   X   A , \ecent  X  \va  + \ecent \vx   A )
}
\newcommand{\tmpe}
{
(\ecent   A , \ecent \va  + \ecent  X ^{-1} \vx   A )
}
\newcommand{\tmpf}
{
(\ecent  A  ,\ecent \va )
}
For $(X,\vx) \in \tgr$,
\[(\ecent, \vzero)(X,\vx)(A,\va)
= (X,\vzero)(\ecent A,\ecent \va)(E,A^{-1}X^{-1}\vx A). \]
Hence
$\tgr (\ecent, \vzero)\tgr(A,\va)\tgr = \tgr(\ecent A,\ecent \va)\tgr$.
Thus
there exists a positive integer $c$ such that 
\[\tilde{T}(\ecent,\vzero)  \tilde{T}( A ,\va )  = c\tilde{T}(\ecent  A , \ecent \va ).\]
Hence we have
\[c = \frac{\deg \tilde{T}(\ecent,\vzero)  \deg \tilde{T}( A ,\va )  }{\deg \tilde{T}(\ecent  A , \ecent \va )} .\]
By Corollary \ref{cor:BasicTheory:Degree:Degree}, we have $c=1$, which completes the proof.
} 
 {
\newcommand{\tmp}
{
	 \grsuba \ast (\ecent \va  \mod \ecent A \abgr )
}
\newcommand{\tmpb}
{
	 \grsuba \ast (\va  \mod   A \abgr )
}
%
}
\end{proof}

\begin{lemma}
{
\newcommand{\tmp}
{
	\left|  \grsuba* (\va   \mod A  \abgr )\right|
}
\newcommand{\tmpb}
{
	\left|  \grsuba*
		(\ecent^{-1} \va  \ecent \mod A   \abgr )
	\right|
}
For all $(A,\va) \in \tmo$,
\[\tilde{T}( A ,\va ) \tilde{T}(\ecent,\vzero) =  \frac{\tmp}{\tmpb}
\tT( \ecent A, \va \ecent).
\] 
}	
\end{lemma}

\begin{proof}
{
\newcommand{\tmpc}
{
(\ecent  ,\vzero)	
}
\newcommand{\tmpb}
{
( X ,\vx )
}
\newcommand{\tmp}
{
( A ,\va ) 
}
\newcommand{\tmpd}
{
(\ecent   A   X  , \va   X  \ecent +   A  \vx   \ecent)
}
\newcommand{\tmpe}
{
(\ecent   A , \va  \ecent +    A  \vx   X ^{-1} \ecent)
}
\newcommand{\tmpf}
{
(\ecent  A  ,\va \ecent )
}
For $\tmpb  \in \tgr$,
 \begin{eqnarray*}
(X,\vx)(\ecent,\vzero) = (\ecent,\vzero)(X,\ecent^{-1} \vx \ecent).
 \end{eqnarray*}
Hence
$\tgr \tmp \tgr \tmpc \tgr 
		=\tgr \tmpf \tgr$.
Thus there exists a positive integer $c$ such that 
\[ \tilde{T}( A ,\va )  \tilde{T}(\ecent,\vzero) = c\tilde{T}(\ecent  A , \va \ecent ).\]
By Corollary \ref{cor:BasicTheory:Degree:Degree}, 
we have
%
%
}
{
\newcommand{\tmp}
{
	\left|  \grsuba *(\va   \mod A  \abgr )\right|
}
\newcommand{\tmpb}
{
	\left|  \grsuba *
		(\ecent^{-1} \va  \ecent \mod  A   \abgr )
	\right|
}
\[c =  \frac{\tmp}{\tmpb},\] which completes the proof.
}
\end{proof}

By the lemma \ref{BasicTheory:RightAction}, the left multiplication by $\tT(C_0,\vzero)$ gives a injective map from the basis $\{\tT(\xi)\}_{\xi \in \tgr \backslash \tmo / \tgr}$ of $\tR$ to the same set.
This implies that $\tT(C_0,\vzero)$ is a left nonzero divisor.
Thus we can define the following element.

\begin{definition}
The linear endomorphism $\theta = \theta_{C_0}$ of $\tR$ is defined by the requirement that 
for each $(A,\va) \in \tmo$
\[\tT{(A,\va)} \tT{(C_0,\vzero)} = \tT{(C_0,\vzero)}\tT{(A,\va)}^{\theta}.\]
\end{definition}

It is clear that $\theta$ is a ring homomorphism.
$\tT{(A,\va)}^{\theta}$ has a following formula.

\begin{corollary}\label{cor:FormulaTheta}
For all $(A,\va) \in \tmo$, 
\[\tilde{T}( A ,\va )^{\theta} =
\frac{
\left|  \grsuba* (\va   \mod A  \abgr )\right|
}
{
\left|  \grsuba* 
	(\ecent^{-1} \va  \ecent \mod A   \abgr )
\right|
}
\tT(A,C_0^{-1} \va C_0).
\] 
Especially, $\xi^{\theta} = \xi$ for each $\xi \in {\rm Im}(s)$.
\end{corollary}
\bigskip

%
%
%
%
{
\renewcommand{\tT}[1]{\tilde{T}\left( #1\right)}
	Let ${\cal X}$ be a subset of $\tmo$ such that $\tgr \bs \tgr{\cal X}\tgr / \tgr$ is a finite set. 
	Then we define  
	\[\tT{{\cal X}}=\sum_{\tgr(A,\va)\tgr \in \tgr \bs \tgr {\cal X} \tgr / \tgr} \tT{A , \va }.\]
	And put 
	\[
	\begin{matrix}
	\tT{{\cal A}, {\mathcal N}}&=&\tT{{\cal A} \times  {\mathcal N}},
	\\  
	\tT{A, {\mathcal N}}&=&\tT{\set{A} \times  {\mathcal N}},
	\\
	\tT{A}&=&\tT{A, \abgr},
	\end{matrix}
	\]
	for a finite subset ${\cal A}$ of $\mo$, $A \in \mo$, and a subset ${\mathcal N}$ of $\abgr$.
By Proposition \ref{BasicTheory:DoubleCoset},
we can define $\tT{A,{\mathfrak a}}$ for ${\mathfrak a} \in \grsuba \bs \abgr /A \abgr$ in an obvious way.  
Then, we easily see  that   $\tT{A, {\mathcal N}}=\sum_{{\mathfrak a}} \tT{A,{\mathfrak a}}$, where ${\mathfrak a}$ runs through
the set \[\gr^A \bs \left( \gr^A{\mathcal N} + A\abgr \right) / A\abgr.\]


}


{
\renewcommand{\tT}[1]{\tilde{T}\left( #1\right)}
\newcommand{\tmp}
{
	f_{\ecent}
}
\newcommand{\tmpb}
{
	\va  \mod  A  \abgr
}
\newcommand{\tmpc}
{
	\ecent^{-1} \va  \ecent \mod  A  \abgr
}
\newcommand{\tmpd}
{
	{\mathcal N}
}
\newcommand{\tmpe}
{
	{\mathcal N}'
}
\newcommand{\tmpf}
{
	\abgr /  A  \abgr 
}
\newcommand{\tmpg}
{
	 \left( A  \abgr + \ecent^{-1} \abgr \ecent \right) /  A  \abgr
}
\newcommand{\tmpi}
{
	{\mathfrak a}  \in  \grsuba \backslash \tmpd_A
}
\newcommand{\tmpj}
{
	 {\mathfrak b}  \in  \grsuba \backslash \tmpe_A
}
\newcommand{\tmpk}
{
	 {\mathfrak a}     \in  \bar{f_A}^{-1}(  {\mathfrak b}   )
}
\newcommand{\tmpl}
{
	\frac{\card{ \varphi_{\tmpd}^{-1}({\mathfrak a})  }}{ \card{ \varphi_{\tmpd'}^{-1}(\bar{f_A}({\mathfrak a} ))  }}
}
\newcommand{\tmpm}
{
	\frac{\card{ \varphi_{\tmpd}^{-1}({\mathfrak a})      }}{\card{ \varphi_{\tmpe}^{-1}({\mathfrak b})   }}
}
\newcommand{\tmpn}
{
	\frac{\card{f_A^{-1}( \varphi_{\tmpe}^{-1}({\mathfrak b})   )}}{\card{ \varphi_{\tmpe}^{-1}({\mathfrak b})   }}
}

\newcommand{\im}
{
	{\rm Im}
}

\begin{proposition}\label{Thetheta:tT(A)theta}
	Let $\tmpd$ be a subgroup of $\abgr$
	such that $\gr^A \tmpd = \tmpd$, $\tmpd + A \abgr = \tmpd$.
	Put $\tmpe = \ecent^{-1} \tmpd \ecent $. 
	Let $\tmpd_A$ and $\tmpe_A$ be the images of $\tmpd$ and $\tmpe$ under the canonical map
	$\abgr \to \abgr / A \abgr$, respectively. 
	Then we have
	\[
	\tT{A,{\cal N}}^\theta =  \frac{\left|\tmpd_A\right|}{\left|\tmpe_A\right|}
	\tT{A,\tmpe}.
	\]
\end{proposition}

	\begin{proof}
	Let $f:\tmpd \to \tmpe$ be the morphism being 
	$(\va \mapsto \ecent^{-1} \va \ecent)$.
	Let $f_A$ and $\bar{f_A}$ be its derived maps 
	$\tmpd_A \to \tmpe_A$
	and
	$\grsuba \backslash \tmpd_A \to \grsuba \backslash \tmpe_A$,
	respectively. 
	Let $\varphi$ be the canonical map $\tmpf \to \grsuba \backslash \tmpf$.
	Put $\varphi_\tmpd = \varphi |_{\tmpd_A}$, $\varphi_{\tmpe} = \varphi |_{\tmpe_A}$.
	Since $\tmpd_A$ and $\tmpe_A$ are $\gr^A$-invariant sets, we have
	$\varphi^{-1}({\mathfrak a}) = \varphi_{\tmpd}^{-1}({\mathfrak a})$ and
	$\varphi^{-1}({\mathfrak b}) = \varphi_{\tmpe}^{-1}({\mathfrak b})$
	for each ${\mathfrak a} \in \grsuba \backslash \tmpd_A$, 
	${\mathfrak b} \in \grsuba \backslash \tmpe_A$.
	Hence
	\begin{eqnarray*}
	\tilde{T}(A,{\cal N})^{\theta} 	&=& \sum_{\tmpi}\tmpl\tilde{T}( A , \bar{f_A}({\mathfrak a}) ). \\ 
	\end{eqnarray*}
Hence
	\begin{eqnarray*}
		\tilde{T}(A,{\cal N})^{\theta}	&=& \sum_{\tmpj} \ 
						\sum_{\tmpk}\tmpm\tilde{T}( A ,     {\mathfrak b}   ).  \\ 
	\end{eqnarray*}
	Since $\bar{f_A} \circ \varphi_{\tmpd} 
		= \varphi_{\tmpe} \circ f_A$,
	\begin{eqnarray*}
		\tilde{T}(A,{\cal N})^{\theta} 
		&=& \sum_{\tmpj} \tmpn\tilde{T}( A , {\mathfrak b}   ).  \\ 
	\end{eqnarray*}
	Since $\card{f_A^{-1}( \varphi_{\tmpe}^{-1}({\mathfrak b})   )} 
=  \card{ \varphi_{\tmpe}^{-1}({\mathfrak b})   }\card{\ker f_A}
= \card{ \varphi_{\tmpe}^{-1}({\mathfrak b})   }\card{\tmpd_A}/\card{\tmpe_A}
$,
	\[\tilde{T}(A,{\cal N})^{\theta}	=  
	\frac{\card{\tmpd_A}}{\card{\tmpe_A}}\sum_{\tmpj}  \tilde{T}( A ,  {\mathfrak b}   ).\]
	This completes the proof.
	\end{proof}
}

%
\subsection{The projection $\phi$}
Let $\ecent$ be an element of the center of $\mo$ satisfying the following assumption:
\begin{assumption}\label{assumption:phi}
$\ecent^{-1} \abgr\ecent \not = \abgr$.
\end{assumption}
\medskip
Since the endomorphism of $\abgr$ being $(\va \mapsto \ecent^{-1} \va \ecent)$ is injective, 
the sequence $\{\ecent^{-n} \abgr \ecent^n\}_n$ is a strictly decreasing sequence of sets.
Since $\abgr/A\abgr$ is a finite set, there exists a positive integer $m$
such that $C_0^{-m} \abgr C_0^m \subset A\abgr$.
Hence the sequence $\set{C_0^{-n} \va C_0^n \mod A \abgr}_n$
is stable and its limit is the identity element for each $\va \in \abgr$. Thus, Proposition \ref{BasicTheory:DoubleCoset} and Corollary \ref{cor:FormulaTheta} imply the stablity 
of the sequence $\{\tT(A,\va)^{\theta^n}\}_n$ for each $(A,\va)\in \tmo$ and that its limit is an element of the image of $s$.
 
\begin{definition}
For each $(A,\va)\in \tmo$, the limit of the sequence $\{\tT(A,\va)^{\theta^n}\}_n$ is denoted by $\tT(A,\va)^{\theta^\infty}$, and
$\phi : \tR \to R$ is defined by $s^{-1} \circ \theta^\infty$.
\end{definition}

Note that both $\theta^\infty$ and $\phi$ are naturally ring homomorphisms.
\smallskip

\begin{proposition}\label{prop:propertymorphisms}
The morphism $\phi$ satisfies the following properties:
\begin{enumerate}
\item $\phi \circ s = id_R$. 
\item $\phi \circ \theta = \phi$.
\item $\tT(A)^\phi = \card{\abgr / A \abgr} T(A).$
\end{enumerate}
\end{proposition}

\begin{proof}
The assertions 1-2 are trivial. 
We shall prove the last assertion. We see that there exists a positive integer $m$
such that $C_0^{-m} \abgr C_0^m \subset A\abgr$.
Put $\eta=\theta_{C_0^m}$, then 
$\tT(A)^{\theta^\infty}=\tT(A)^{\theta^m} = \tT(A)^{\eta}$.
By Proposition \ref{Thetheta:tT(A)theta}, we have 
$\tT(A)^{\eta} = \card{\abgr / A \abgr} \tT(A,\vzero)$, which completes the proof.
\end{proof}
\medskip


\subsection{Computation of the product $\tT(A,\vzero) \tT(B)$}
\label{section:CommutativeCase}
{

\newcommand{\eqna}[1]{
	\begin{eqnarray*}
		#1
	\end{eqnarray*}
}
\newcommand{\calY}{{\cal Y}}
\newcommand{\calX}{{\cal X}}
In this subsection, we compute the product $\tT(A,\vzero) \tT(B)$ for each
$A,B \in \mo$, in the case where $G$ has an anti-automorphism $(\alpha \mapsto \hat{\alpha})$ 
preserving $\gr A \gr$ for all $A \in \mo$.
Then it is well known that $R$ is commutative.
For each $A, B ,C \in \mo$, we put \[\calX (A,B,C) =\left\{ \alpha \in \gr A \gr  \ | \ \hat{\alpha}^{-1}C \in  \gr B \gr \right\}.\]
Then, for each system of representatives $\{\alpha_i\}_i$ of $\gr \bs \gr A \gr $ in $\gr A \gr$, $\{\hat{\alpha_i}\}_i$ is a system of representatives of $\gr A \gr / \gr $ in $\gr A \gr$.
Hence we see that, for each $A,B \in \mo$,
 \[T(A)T(B) = \sum_{\gr C \gr \in \gr \bs \mo / \gr} \card{ \gr \bs \calX(A,B,C)}T(C).\]
Similarly, we have the following formula:

\begin{lemma}\label{BasicTheory:lemma:tTAzerotTB}
For $A,B,C \in \mo$, and $\vc \in \abgr$, we put \[\calY(A,B,C,\vc) = \set{\alpha \in \calX(A,B,C) \ | \  \hat{\alpha} ^{-1}\vc \in \abgr}.\]
Then, for $A, B \in \mo,$ we have
\eqna{
	\tT(A,\vzero) \tT(B) &=& 
	\sum_{(C,\vc)}
	\card{ \gr \bs \calY(A,B,C,\vc)}
	\tT(C,\vc),
}
where $(C,\vc)$ runs through a system of representatives of 
$\tgr \bs \tmo / \tgr$ in $\tmo$.
\end{lemma}

\begin{proof}
Let $\{\alpha_i\}_i$ be a system of representatives of $\gr \bs \gr A \gr $ in $\gr A \gr$.
Since $A\abgr \supset \abgr A$, the set $\{(\hat{\alpha_i}, \vzero)\}$ is a system of representatives of 
$\tgr (A,\vzero) \tgr / \tgr$ in $\tgr (A,\vzero) \tgr$.
Hence, we have
\eqna{\tT(A,\vzero) \tT(B) 
&=& 
\sum_{(C,\vc)}\card{\left\{ i  \ | \ {(\hat{\alpha_i}}^{-1},\vzero) (C,\vc) \in \cup_{\vb \in \abgr} \tgr (B,\vb) \tgr \right\}}
\tT(C,\vc)
\\
&=& 
\sum_{(C,\vc)}\card{\left\{ i  \ | \ {\hat{\alpha_i}}^{-1} C \in \gr B \gr, \hat{\alpha_i}^{-1} \vc \in \abgr \right\}}
\tT(C,\vc)
\\
&=&
\sum_{(C,\vc)}
\card{ \gr \bs \calY(A,B,C,\vc)}
\tT(C,\vc).
}
\end{proof}
}
%
%

}
\section{Local Hecke rings and local Hecke series associated with algebras}\label{section:RingsOfHeckeAlgebras}
{
\newcommand{\TL}[1]{T_L\left( #1\right)}

In this section, we recall the definition of the local Hecke rings and the local Hecke series defined in \cite{Hy2} and \cite{Hy4}, respectively.
Let $L$ be an algebra which is free of rank $r$ as an abelian group, and fix a $\Z$-basis of $L$. Then $\Aut_{\Q_p}^{alg}(L \otimes \Q_p)$ and $ \End_{\Z_p}^{alg}(L \otimes \Z_p)$ are identified with subsets of $M_r(\Q_p)$. The following notation were introduced.
\[
	\begin{array}{llllll}
		\glp &=& \Aut_{\Q_p}^{alg}(L \otimes \Q_p), \\ \rule{0ex}{3ex}
		 \dlp &=& \End_{\Z_p}^{alg}(L \otimes \Z_p) \cap \glp, \\ \rule{0ex}{3ex}
		\gamlp  &=& \Aut_{\Z_p}^{alg}(L \otimes \Z_p).
	\end{array}
\]
The local Hecke ring $\rlp$ associated with $L$ is defined by the Hecke ring with respect to the pair $(\gamlp,\dlp)$.

Next, the local Hecke series $\plp(X)$ defined in \cite{Hy4} are recalled by using notation of Section \ref{section:BasicTheory}.
Set $\tlp = T_{\gamlp,\dlp}$. 
For each nonnegative integer $k$, define an element $\tlp(p^k)$ of $\rlp$ by
 \[\sum_\alpha \tlp(\alpha),\] where $\alpha$ runs through a system of representatives of $\gamlp \bs \dlp / \gamlp$ and satisfies $\left| \lp/\lp^{\alpha} \right| = p^k$.
The local Hecke series $\plp(X)$ associated with $L$ is defined to be the generating function of 
$\{\tlp(p^k)\}_k$, that is, 
\[\plp(X)  = \sum_{k \geq 0}\tlp(p^k)X^k \in \rlp[[X]].\] 

\begin{example}
\ 
\begin{enumerate}
\item If $L$ is the ring $\Z^r$ of the direct sum of $r$-copies of $\Z$ for some positive integer $r$, then $\gamlp = GL_r(\Z_p)$ and $\dlp = M_r(\Z_p) \cap GL_r(\Q_p)$.
The ring structure of $\rlp$ and the rationality of $\plp(X)$ were shown by \cite[Hecke]{He} and \cite[Tamagawa]{T}.
\item If $L$ is the Heisenberg Lie algebra of dimension $3$, $\rlp$ and $\plp(X)$ are treated in our previous paper \cite{Hy}. 
$\plp(X)$ is slightly different from $D_{2,2}(X)$ defined in \cite{Hy}. 
For the details, see Remark \ref{remark:Dnini(X)}.
\end{enumerate}
\end{example}
}

\section{The local Hecke rings associated with the Heisenberg Lie algebras}
\label{section:Heisenberg Lie algebras}
{
\newcommand{\tT}{\tilde{T}}

\newcommand{\TL}[1]{T_L\left( #1\right)}
\newcommand{\pmat}[1]
{
		\begin{pmatrix}
			#1
		\end{pmatrix}
}
\newcommand{\mat}[1]
{
		\begin{matrix}
			#1
		\end{matrix}
}

\newcommand{\eqna}[1]{
	\begin{eqnarray*}
		#1
	\end{eqnarray*}
}

\newcommand{\set}[1]{\left\{ #1 \right\}}
\newcommand{\card}[1]{\left| #1 \right|}

\newcommand{\ep}
{{\langle p \rangle}}
\newcommand{\epp}
{{\langle p^2 \rangle}}
Let $n$ be a positive integer.
The Heisenberg Lie algebra of dimension $2n + 1$ over $\Z$, denoted by $\Hn$, 
is the Lie algebra of square matrices of size $n+2$ with entries in $\Z$ of the form
\[
\begin{pmatrix}
	0 & \va & c\\
		0 &O_n & \vb \\ 
	0  & 0  &   0\\
\end{pmatrix},
\]
where
$\va$ is a row vector of length $n$, $\vb$ is a column vector of length $n$,
and $O_n$ is the zero matrix of size $n$.
We study the local Hecke ring and the local Hecke series associated with $\Hn$.
Put $L = \Hn$.
By choosing the standard basis of $L$, one can identify $\glp$ with the group of matrices of the form
\[
\pmat{
	A & \va \\
	\vzero_{2n} & \mu (A)
},
\]
where $A \in GSp_{2n}(\Q_p)$, $\mu (A)$ is the multiplier of $A$,  
$\va$ is a column vector of length $2n$,
and $\vzero_{2n}$ is the zero row vector of length $2n$.
It is easy to see that $\dlp = \glp \cap M_{2n + 1}(\Z_p)$ and $\gamlp = \glp \cap GL_{2n + 1}(\Z_p).$  
\smallskip

Put $G = GSp_{2n}(\Q_p)$, $\mo = G \cap M_{2n}(\Z_p)$, 
$\gr = GSp_{2n}(\Z_p)$, 
and let $\abgr$ be the set of column vectors of length $2n$ with entries in $\Z_p$. 
Then $\mo$ acts on $\abgr$ on the left naturally, and does on the right as follows:
$\va A = \mu(A) \va$ for each $\va \in \abgr$, $A \in \mo$.
Hence  $\abgr$ is a both sides $\mo$-module. Note that $X*\va = \mu(X)^{-1}X\va$ for each $A \in \mo$, $X \in \gr^A$, $\va \in \abgr$.
Let us put $\tmo = \tmo({\mo,\abgr}),~\tgr = \tgr({\gr,\abgr})$. 
It is well known that $(\gr,\mo)$ is double finite (cf.\cite{S2}). 
Thus, Corollary \ref{cor:DoubleFiniteness} implies that $(\tgr,\tmo)$ is also double finite.
We put $R = R(\gr,\mo)$, $T = T_{\gr,\mo}$, $\tR = R(\tgr,\tmo)$ and $\tT = T_{\tgr,\tmo}$.
Then it is easy to see that $\dlp$ is isomorphic to $\tmo$ 
by the map
\[
\pmat{
	A & \va \\
	\vzero_n & \mu (A)
}
\mapsto
(A, \va),
\]
and that the isomorphism derives the isomorphism from $\rlp$ onto $\tR$.
Let us indentify $\rlp$ with $\tR$. Then
\[\tlp\left(\pmat{
	A & \va \\
	\vzero_n & \mu (A)
}
\right)
=
\tT(A,\va)
.
\] 

Let $k$ be a nonnegative integer.
We denote by $T(p^k)$ the sum of all the elements of the form $T(\gr A \gr)$ with $\gr A \gr \in \gr \bs \mo / \gr$ and $v_p(\mu(A)) = k$, and by $\tT(p^k)$ the sum of all the elements of the form $\tT(\tgr (A,\va) \tgr)$ 
with $\tgr (A,\va) \tgr \in \tgr \bs \tmo / \tgr$ and $v_p(\mu(A)) = k$.
The Hecke series $P_n(X)$ associated with $GSp_{2n}$ is defined by 
\[P_n(X) = \sum_{k \geq 0}T(p^k)X^k.\]
In addition, let us define the formal power series $\tP_n(X)$ with coefficients in $\tR$ by 
\[\tP_n(X) = \sum_{k \geq 0}\tT(p^k)X^k.\]

Note that $\tT(p^k) = \tlp(p^{(n+1)k})$ for all nonnegative integer $k$, and
$\tP_n(X^{n+1}) = \plp(X)$.

\begin{remark}\label{remark:Dnini(X)}
$D_{2,2}(X)$ in our previous paper \cite{Hy} coincides with $\tP_1(X)$. Thus $\pHop(X)=D_{2,2}(X^2)$.  
\end{remark}
\medskip

{
\newcommand{\ecent}{C_0}

Put $\ecent = pE$. Then, the tuple $(G,\mo, \gr, \abgr, C_0)$
satisfies Assumptions \ref{assumption:deg}, \ref{assumption:s} and \ref{assumption:phi}. Hence, the ring homomorphisms, 
\[s=s_n:R \to \tR, \  \theta=\theta_n:\tR \to \tR, \ \text{and}\  \phi=\phi_n:\tR \to R.\]
are constructed.
Set $\ep=T(pE)$, $\epp = T(p^2E)$.
Note that for each $(A,\va) \in \tmo$, we have
\[\ep^s \tT{(A,\va)} = \tT{(pA,p\va)},\] and
\[\tilde{T}( A ,\va )^{\theta} =
\frac{
\left|  \gr^A* (\va   \mod A  \abgr )\right|
}
{
\left|  \gr^A* 
	(p\va \mod A   \abgr )
\right|
}
\tT(A,p \va).
\] 
}
By Proposition \ref{prop:propertymorphisms}, we have the following properties:

\begin{proposition}\label{prop:Relation:Heisenberg}
The three ring homomorphisms $\phi$, $s$ and $\theta$ satisfy the following properties:
\begin{enumerate}
\item
$\phi \circ s = id_R$. 
\item
$\phi \circ \theta = \phi$.
\item
$\tT(A)^\phi = p^{n \cdot v_p(\mu(A))} T(A)$.
\end{enumerate}
\end{proposition}

\begin{proof}
We need only prove the last identity, 
which is an immediate consequence of the facts $\card{\abgr / A \abgr} = p^{v_p(\det A)}$ and $(\det A)^2 = \mu(A)^{2n}$.
\end{proof}

\begin{corollary}
$\tT(p^{k})^{\phi} = p^{nk}T(p^k)$ for each nonnegative integer $k$.
\end{corollary}

\begin{corollary}\label{corollary:phiDtD}
$\tP_n^{\phi}(X) = P_n(p^{n}X)$.
\end{corollary}

\begin{proof}
	They are straightforward consequences of Property $3$.
\end{proof}
\bigskip

Next we show the non-commutativety of the coefficients of $\tP_n(X)$. 
For every positive integer $m$, we denote by $\tmo_{p^m}$, the set of elements $(A,\va)$ of 
$\tmo$ such that $v_p(\mu(A))=m$.
To prove this, we need the following lemma.

\begin{lemma}
For positive integers $k$, $l$ with $k < l$. Let $\vc_2$ be an element of $\Zp^n - p\Zp^n$, and put 
\[C = \pmat{p^kE & O\\O & p^lE}\in \tmo, \ \vc = \pmat{\vzero \\ \vc_2}\in \Zp^{2n},\]
where $E$ and $O$ are respectively the identity matrix and the zero matrix of size $n$.
Then we have:
\[\tmo_{p^k} \cdot \tmo_{p^l} \ni (C,\vc),\ \tmo_{p^l} \cdot \tmo_{p^k} \not \ni (C,\vc).\]
\end{lemma}

\begin{proof}
Since
\[\left(\pmat{p^kE & O \\ O & E}, \vzero \right) \in \tmo_{p^k},\ 
\left(\pmat{E & O \\ O & p^lE}, \pmat{\vzero \\ \vc_2}\right) \in \tmo_{p^l},\]
we have $\tmo_{p^k} \cdot \tmo_{p^l} \ni (C,\vc).$

Let $(A,\va)$ and  $(B,\vb)$ be  elements of $\tmo_{p^l}$ and $\tmo_{p^k}$, respectively.
And assume $AB = C$. The second component of  $(A,\va)(B,\vb)$ is contained in
$A\vb + p^k\Zp^{2n}$. 
Since $A = CB^{-1}$ and $p^kB^{-1} \in \tmo$, we have $A\Zp^{2n} \subset \pmat{\Zp^{n} \\ p^{l-k}\Zp^n}$.
Hence $A\vb + p^k\Zp^{2n}$ is contained in $\pmat{\Zp^{n} \\ p\Zp^n}$. 
Therefore, $\tmo_{p^l} \cdot \tmo_{p^k} \not \ni (C,\vc).$
\end{proof}
\medskip

Now we prove the non-commutativity.

\begin{theorem}\label{thm:noncommutativity}
Notations being as the above lemma, $\tT(p^k)$ and $\tT(p^l)$ are not commutative to each other. 
\end{theorem}

\begin{proof}
By the above lemma, the coefficients of $\tT(C,\vc )$ in 
$\tT(p^k)\tT(p^l)$ and $\tT(p^l)\tT(p^k)$ are respectively not zero and zero, respectively.
\end{proof}
\bigskip

To state the identity for $\tP_n(X)$, we recall the rationality theorem associated with $GSp_{2n}$.

\begin{theorem}[Rationality Theorem, \cite{A1}, \cite{He}, \cite{S1}]
\label{theorem:rationality}
$P_n(X)$ is rational, namely
there exist elements $q_1,...,q_{2^n}$ of $R$ and $g(X) \in R[X]$ such that
\begin{equation}\label{HAeq}
\sum_{i\geq 0} q_i X^i P_n(X)  = g(X),
\end{equation}
where $q_0 = 1$.
Especially, 
\begin{enumerate}
\item If $n=1$ then $q_1 = -T(\diag(1,p))$, $q_2 = p\ep$, $g(X) = 1$.
\\
\item If $n=2$ then 
\eqna{
q_1 &=& -T(\diag(1,1,p,p)),\\
q_2 &=& pT(\diag(1,p,p^2,p)) + p(p^2+1)\ep,\\ 
q_3 &=& -p^3\ep T(\diag(1,1,p,p)),\\ 
q_4 &=& p^6\epp,\\ 
g(X) &=& 1 - p^2 \ep X^2.
}
\end{enumerate}
Where, $\ep=T(pE)$ and $\epp = T(p^2E)$.
\end{theorem}

Note that the sequence $\{q_i\}_i$ and $g(X)$ depend on $n$.
Our previous paper \cite{Hy} shows the following theorem:
\begin{theorem}[{\cite[Theorem 7.8]{Hy}}]
Let us keep the notations of Theorem \ref{theorem:rationality}.
If $n=1$ then $\tP_1(X)$ satisfies the following identity.
\[\tP_1^{\theta^2}(X) - q^s_1 Y\tP_1^{\theta}(X)   + q^s_2 Y^2\tP_1(X) = 1,\]
where $Y = pX$. 
\end{theorem}
\medskip

We note that it derives the rationality of $P_1(X)$ using Corollary \ref{corollary:phiDtD}.
Moreover, the above identity is a solution of the following problem in the case where $n=1$.

\begin{problem}\label{myprob}
Let us keep the notations of Theorem \ref{theorem:rationality}.
Does $\tP_n(X)$ satisfy the following formula for each $n$ ?
\begin{equation}
\label{equation:general:main}
\sum^{2^n}_{k=0} q^s_k Y^k \tP_n^{\theta^{2^n-k}}(X) = g^s(Y),
\end{equation}
where $Y = p^nX$.
\end{problem}

Remark that,  for each $n$, this identity implies the rationality of $P_n(X)$ via the morphism $\phi_n$.
In the present paper, we solve our problem for the case $n=2$.
\bigskip

\begin{proposition}
 Equality (\ref{equation:general:main}) is true modulo $X^{2^n}$. 
\end{proposition}

\begin{proof}
Let $h(X) = \sum_{k=0}^\infty a_k X^k$ be the left hand side of (\ref{equation:general:main}).
Then, for each $k \leq 2^n$, 
$a_k$ is a linear combination of
elements of $\{\tT\left(p^l\right)^{\theta^m}\}_{l \leq m}$ with coefficients in the image of $s$.
Hence for each $k \leq 2^n$, $a_k$ is an element of the image $s$.
The proposition follows from the injectivity of $s$ and the Theorem \ref{theorem:rationality}.
\end{proof}

Thus we are reduce to proving the following problem.

\begin{problem}Put $N=2^n$.
For all $k \geq 0$, Is the following identity true?
\[\displaystyle \sum_{i=0}^N p^{n(N-i)}q^s_{N-i}\tT\left(p^{k+i}\right)^{\theta^i}=0.\]

\end{problem}

\newcommand{\calY}{{\cal Y}}
\newcommand{\calX}{{\cal X}}

In the rest of this section, we give a formula of the product $\tT(A,\vzero) \tT(p^k)$ using Lemma \ref{BasicTheory:lemma:tTAzerotTB}.
Define $\hat{A} = \mu(A)A^{-1}$ for each $A \in G$. Then the anti-automorphism $(A \mapsto \hat{A})$ satisfies the condition in Subsection \ref{section:CommutativeCase}.
Since
\eqna{
\calX (A,B,C) 
&=&\left\{ \alpha \in \gr A \gr  \ | \ \alpha C \in  \mu(A)\gr B \gr \right\},\\ 
\calY(A,B,C,\vc) 
&=& \set{\alpha \in \calX(A,B,C) \ | \  \alpha \vc \in \mu(A)\abgr},
}
 for all $A,B,C \in \mo$, $\vc \in \abgr$, we have:
 
\begin{lemma}\label{HeisenbergLiealgebras:lemma:tTAzerotTB}
For each $A, C \in \mo$, $\vc \in \abgr$, put 
\eqna{
\calX (A,C) &=&\left\{ \alpha \in \gr A \gr  \ | \ \alpha C \in  \mu(A)\mo \right\},\\
\calY(A,C,\vc) &=& \set{\alpha \in \calX(A,C) \ | \ \alpha \vc \in \mu(A)\abgr}.
}
Then, for each $A\in \mo$,
\eqna{\tT(A,\vzero) \tT(p^k) &=& 
\sum_{(C,\vc)}
\card{ \gr \bs \calY(A,C,\vc)}
\tT(C,\vc),}
where $(C,\vc)$ runs through a system of representatives of 
$\tgr \bs \tmo / \tgr$ and
satisfies $v_p(\mu\left(A^{-1}C)\right) = k$.
\end{lemma}

}
\section{The local Hecke series associated with the Heisenberg Lie algebra of dimension 5}
{
\label{section:MainResult}
\newcommand{\F}{\mathbb F}
\newcommand{\Fp}{{\mathbb F}_p}

\newcommand{\set}[1]{\left\{ #1 \right\}}
\newcommand{\card}[1]{\left| #1 \right|}

\newcommand{\Tmat}{\mat{0 & 1 \\ 1 & 0}}

\newcommand{\tT}[1]{\tilde{T}\left( #1\right)}

\newcommand{\ep}
{{\langle p \rangle}}
\newcommand{\epp}
{{\langle p^2 \rangle}}
\newcommand{\tep}
{\ep^s}
\newcommand{\tepp}
{\epp^s}

\newcommand{\Imat}{I}
\newcommand{\Zmat}{O}

\newcommand{\diagmat}[4]
{
	\pmat{
		#1 & 0 & 0 & 0\\
		0 & #2 & 0 & 0\\
		0 & 0 & #3 & 0\\
		0 & 0 & 0 & #4\\
	}
}

\renewcommand{\vec}[4]
{
	\pmat{
		#1\\
		#2\\
		#3\\
		#4\\
	}
}
\newcommand{\pmat}[1]
{
		\begin{pmatrix}
			#1
		\end{pmatrix}
}
\newcommand{\mat}[1]
{
		\begin{matrix}
			#1
		\end{matrix}
}
\newcommand{\eqna}[1]{
	\begin{eqnarray*}
		#1
	\end{eqnarray*}
}

{
In this section, we shall prove our main theorem.
Let us keep the notation of Section \ref{section:Heisenberg Lie algebras},
and suppose $n=2$. First we introduce some notation.
\[
\ve_1 = \vec{1}{0}{0}{0}, 
\ve_2 = \vec{0}{1}{0}{0}, 
\ve_3 = \vec{0}{0}{1}{0},
\] 
\eqna{
P_{1}&=&
\diag(1,1,p,1),\\
P_{3}&=&
\diag(1,p,p,p),\\
A&=&
\diag(1,1,p,p),\\
B &=& P_1P_3 = 
\diag(1,p,p^2,p),\\
		\gr^A&=& \gr \cap A \gr A^{-1} = 
			\set{
			\pmat{	*	& * 	&* 	&* \\
				*	&*	&*	&* \\
				p*	&p* 	&* 	&* \\
				p*	&p*	&*	&* \\
			}
			}
			\cap
			\gr,\\
		\gr'&=& \gr^{P_1} = \gr^{P_3} =
			\set{
			\pmat{	*	& * 	&* 	&* \\
				p*	&*	&*	&* \\
				p*	&p* 	&* 	&p* \\
				p*	&*	&*	&* \\
			}
			}
			\cap
			\gr.		
}
\medskip

Next, we provide a lemma.

\begin{lemma}\label{Caluculation:lemma:Decomposition}
The following identities hold.
\begin{eqnarray}
\label{eq:grprime_e3}
\gr' \ve_3 &=& \Zp^4-P_1\Zp^4,\\
\label{eq:grprime_e2}
\gr' \ve_2 &=& P_1\Zp^4 - P_3\Zp^4,\\
\label{eq:grprime_e1}
\gr' \ve_1 &=& P_3\Zp^4 - p\Zp^4,\\
\label{eq:grA_e3}
\gr^{A}\ve_3 &=& \Zp^4 - A\Zp^4,\\
\label{eq:grA_e1}
\gr^{A}\ve_1 &=& A\Zp^4 - p\Zp^4.
\end{eqnarray}
\end{lemma}

\begin{proof}
It is essential to prove that the left-hand-side contains the right-hand-side for each identity.
Put 
\[
P(x,y,z)=
\pmat{
	1 & -z & x & y\\
	0 & 1 & y & 0 \\
	0 & 0 & 1 & 0 \\
	0 & 0 & z & 1 \\ 
},\ 
Q(x,y)=
\pmat{
	1 & x & 0 & 0\\
	0 & 1 & 0 & 0 \\
	0 & y & 1 & 0 \\
	y & 0 & -x & 1 \\ 
}
\]
for each $x,y,z \in \Zp$.
Clearly, they are elements of $\gr$. 
Let \[\va = \vec{a}{b}{c}{d}\] be an element of $\Zp^4$.
\begin{enumerate}
\renewcommand{\labelenumi}{(\arabic{enumi}).}
\setcounter{enumi}{\arabic{equation}}
\addtocounter{enumi}{-5}
\item
For $\va  \in \Zp^4-P_1\Zp^4$ , we show
$\gr'\va = \gr' \ve_3$.
We are reduced to the case where $c = 1$.
Then 
\[
\va = P(a,b,d)\ve_3.
\] 
\item 
We assume $\va \in P_1\Zp^4 - P_3\Zp^4$. 
For $X \in SL_2(\Zp)$,
\[
T_{2,3}
\pmat{
		E & O\\
		O & X
		}T_{2,3}
\]
is an element of $\gr'$, where $T_{2,3}$ is the permutation matrix corresponding to $(2 \ 3) \in S_4$.
Hence we are reduce to the case where $b=1$ and $d=0$. Then
\[
\va =Q(a,c)\ve_2.
\]
\item
For $\va \in P_3\Zp^4 - p\Zp^4$, 
we show $\gr'\va = \gr' \ve_1$.
We are reduced to the case where $a = 1$. Then
\[
\va = \ ^tP(c,d,-b)\ve_1.
\]
\item
For $\va \in \Zp^4 - A\Zp^4$, 
we show $\gr^A \va = \gr^A \ve_3$.
Since 
\[
\pmat{X & O \\ O &  \ ^t X^{-1}} \in \gr^{A}
\] for each $X \in GL_2(\Zp)$,
we may assume that $c=1,d=0$. Then
\[\va = 
P(a,b,0)
\ve_3
\]
\item
For $\va \in  A\Zp^4 - p\Zp^4$ , 
we show $\gr^A \va = \gr^A \ve_1$.
We may assume that $a=1,b=0$ for the same reason. Then
\[\va = 
~^tP(c,d,0)
\ve_1.
\]
\end{enumerate}
This completes the proof.
\end{proof}

\begin{corollary}\label{Caluculation:cor:Decomposition}
The following identities hold.
	\eqna{ 
		 \Z_p^4
		&=&
		\gr'
		\ve_3
		\cup
		\gr'
		\ve_2		
		\cup
		\gr'
		\ve_1
		\cup
		p\Zp^4
\quad (disjoint)
.	
\\
\\
		 \Z_p^4
		&=&
		\gr^{A}
		\ve_3
		\cup
		\gr^{A}
		\ve_1
		\cup
		p\Zp^4
\quad (disjoint)
.\\
}
\end{corollary}

\begin{proof}
Clear.
\end{proof}
Next we give systems of representatives of $\gr \bs \gr A \gr$ in $\gr A \gr$ and $\gr \bs \gr B \gr$ in $\gr B \gr$, respectively.
It is an immediate consequence of
\cite[Proposition 3.35]{A3} and the following lemmas.
The details are left to the reader.
We denote the matrix
$\diag({p^\alpha},
	{p^{\beta}},
	{p^{k-\alpha}},
	{p^{k- \beta}})
$
by $C(\alpha,\beta,k)$ for
nonnegative integers $k$, $\alpha$, $\beta$, with
 $0 \leq \alpha \leq \beta \leq k - \beta$.

\begin{lemma}[{\cite[Chapter 6, Lemma 5.2]{K}}]
	Let $M$ be an element of $\mo$ with $v_p(\mu(M)) = 2$. 
	$M \in \gr B \gr$ if and only if $rk_p(M) = 1$,
	where $rk_p(M)$ means the rank $M$ over $\F_p$.
\end{lemma}

\begin{proof}
It is a straightforward application of the ``symplectic divisors theorem''.(cf. \cite[Theorem 3.28]{A3}).
\end{proof}

\begin{lemma}
	For $D \in GL_n(\Zp)$, we define the set ${\cal B}(D)$ by
	\[
		{\cal B}(D)=\set{B \in M_n(\Zp) \ | \ ^tDB \ \text{is symmetric}}.
	\]
	Then ${\cal B}(DD') ={\cal B}(D)D'$ for each $D' \in GL_n(\Zp)$.
\end{lemma}

\begin{proof}
Obvious.
\end{proof}

\begin{proposition}\label{Calculation:prop:ggAg}
We put
\eqna{
A_1 &=& 
	\pmat{
		p & 0 & 0 & 0 \\
		0 & p & 0 & 0 \\
		0 & 0 & 1 & 0 \\
		0 & 0 & 0 & 1 \\
	}
,\\
A_2 (d,x)&=& 
	\pmat{
		p & 0 & 0 & 0 \\
		-x & 1 & 0 & d \\
		0 & 0 & 1 & x \\
		0 & 0 & 0 & p \\
	}
\ \text{with} \ 0 \leq d,x < p
,\\
A_3\pmat{a & b \\b & d \\}&=& 
	\pmat{
		1 & 0 & a & b \\
		0 & 1 & b & d \\
		0 & 0 & p & 0 \\
		0 & 0 & 0 & p \\
	}
\ \text{with} \ 0 \leq a, b ,d <p,\\
A_4(d) &=& 
	\pmat{
		1 & 0 & d & 0 \\
		0 & p & 0 & 0 \\
		0 & 0 & p & 0 \\
		0 & 0 & 0 & 1 \\
	}
\ \text{with} \ 
	0 \leq d < p
.
}
Then the set of these matrices is a system of representatives of
$\gr \bs \gr A \gr$ in $\gr A \gr$.

\end{proposition}

\begin{corollary}\label{Calculation:cor:ggAg}
Let $\alpha,\beta,k$ be nonnegative integers with $0 \leq \alpha \leq \beta \leq k- \beta$ and $k >0$, then the set of the following matrices is a system of representatives of $\gr \bs {\mathcal X}\left(A,C(\alpha,\beta,k)\right)$ in each case.
\begin{enumerate}
\item If 
	$\alpha = 0,\beta=0$,
	\[A_1.\]
\item If
	$\alpha = 0,\beta \geq 1$,
	\eqna{ 
	A_1, & & \\ 
	A_2(d,0) & & ( 0 \leq d < p).
	}
\item If
	$\alpha \geq 1$,
	all matrices in Lemma $\ref{Calculation:prop:ggAg}$.
	
\end{enumerate}

\end{corollary}

\begin{proof}
Clear.

\end{proof}
}

{
\begin{proposition}\label{Calculation:prop:ggBg}
We put
\eqna{
B_1(x) &= &
	\pmat{
		p^2 & 0 & 0 & 0 \\
		0 & p & 0 & 0 \\
		0 & 0 & 1 & 0 \\
		0 & 0 & 0 & p \\
	}
	\pmat{
		1 & 0 & 0 & 0 \\
		-x & 1 & 0 & 0 \\
		0 & 0 & 1 & x \\
		0 & 0 & 0 & 1 \\
	}
\ \text{with} \ 0 \leq x < p
,\\
B_2 &= &
	\pmat{
		p & 0 & 0 & 0 \\
		0 & p^2 & 0 & 0 \\
		0 & 0 & p & 0 \\
		0 & 0 & 0 & 1 \\
	}
,\\
B_3\pmat{a & b \\ b & d}&= &
	\pmat{
		p & 0 & a & b \\
		0 & p & b & d \\
		0 & 0 & p & 0 \\
		0 & 0 & 0 & p \\
	}
\ \text{with} \ 0 \leq a, b ,d <p, \ rk_p\pmat{a & b \\ b & d} = 1 
,\\
B_4(b,d,x)&= &
	\pmat{
		p & 0 & 0 & pb \\
		0 & 1 & b & d \\
		0 & 0 & p & 0 \\
		0 & 0 & 0 & p^2 \\
	}
	\pmat{
		1 & 0 & 0 & 0 \\
		-x & 1 & 0 & 0 \\
		0 & 0 & 1 & x \\
		0 & 0 & 0 & 1 \\
	}
\ \text{with} \ 0 \leq d < p^2, \ 0 \leq b, x < p
,\\
B_5(a,b)&= &
	\pmat{
		1 & 0 & a & b \\
		0 & p & pb & 0 \\
		0 & 0 & p^2 & 0 \\
		0 & 0 & 0 & p \\
	}
\ \text{with} \ 0 \leq a < p^2, \ 0 \leq b < p
.\\
}Then the set of these matrices is a system of representatives of
$\gr \bs \gr B \gr$ in $\gr B \gr$.

\end{proposition}

\begin{corollary}\label{Calculation:cor:ggBg}
Let $\alpha,\beta,k$ be nonnegative integers with $\alpha \leq \beta \leq k - \beta$ and $k \geq 3$,
then the set of the following matrices is a system of representatives of $\gr \bs  {\mathcal X}\left(B,C(\alpha,\beta,k)\right)$.
\begin{enumerate}
\item If 
	$\alpha = 0,\beta=0$,
	\[\emptyset.\]
\item If 
	$\alpha = 0,\beta \geq1$,
	\[B_1(0).\]
\item If
	$\alpha = 1,\beta =1$,
	\eqna{B_1(x)& &( 0 \leq x < p), \\B_2, \\
		B_3\pmat{a & b \\ b & d}  
		& & \left(0 \leq a, b ,d <p, \ rk_p\pmat{a & b \\ b & d} = 1\right).
	}
\item If
	$\alpha = 1, \beta \geq 2$,
	\eqna{B_1(x) & & ( 0 \leq x < p),\\ B_2, \\
		 B_3\pmat{a & b \\ b & d}  
			& & \left(0 \leq a, b ,d <p, \ rk_p\pmat{a & b \\ b & d} = 1\right),\\
		B_4(b,d,0)& & (0 \leq d < p^2, \ 0 \leq b< p).
		}
\item If
	$\alpha \geq 2$,
	all matrices in Lemma $\ref{Calculation:prop:ggBg}$.
	
\end{enumerate}

\end{corollary}

\begin{proof}
Obvious.

\end{proof}
}

Next we calculate $\tT{A,\vzero}\tT{p^{k+1}}^{\theta}$ and $\tT{B,\vzero}\tT{p^{k+2}}^{\theta^2}$. 
Let us introduce some notation. Put
\[{\cal C}_k = \set{\left. C(\alpha,\beta,k)\right| 0 \leq \alpha \leq \beta \leq k-\beta},\ 
{\cal C}_k^0 = \set{\left. C(0,\beta,k)\right| 1 \leq \beta \leq k - \beta}.\]
For each $D \in \mo$ and each finite subset ${\cal C}$ of  $\mo$, we put
\[S_D({\cal C}) = \sum_{C \in {\cal C}}\sum_{\va}\left|\gr\bs{\cal Y}(D,C,\va)\right| \tT{C,\va},\]
where $\va$ runs through a system of representatives of 
$\gr^{C} \bs \Zp^4 /C \Zp^4$ for each $C$.
The following lemmas are useful.

\begin{lemma}\label{Calculation:lemma:DisjointSum}
Let ${\cal N}$ and ${\cal N'}$ be two subsets of $\Z_p^4$, and let $C \in \mo$.
If \  $\gr^C{\mathcal N} + C\Z_p^4$ and $\gr^C{\mathcal N}' + C\Z_p^4$ are disjoint,
then
	 \[\tT{C,{\mathcal N} \cup {\mathcal N'}} 
		= \tT{C,{\mathcal N}} + \tT{C,{\mathcal N'}}.\]
Especially, if ${\cal N}$ is a subset of $p\Z_p^4$, ${\cal N'}$ is a subset of $\Z_p^4 - p\Z_p^4$,
and $C \in p\mo$, then the above formula holds. 
\end{lemma}

\begin{proof}
Clear.
\end{proof}

\begin{lemma}\label{lemma:Thetheta}
For each $k \geq 0$, the following identities hold:
\begin{eqnarray}
%
\label{lemma:Thetheta:eq1}
\tT{p{\cal C}_{k}, \Zp^4}^\theta &=& p^4\tep \tT{p^k},\\
\label{lemma:Thetheta:eq2}
\tT{pA^{k+2}, A\Zp^4}^\theta
&=& p^2\tT{pA^{k+2}, p\Zp^4}^{\theta},\\
\label{lemma:Thetheta:eq3}
\tT{{\cal C}_{k+4}^0, P_1\Zp^4}^{\theta} &=& p^3\tT{{\cal C}_{k+4}^0, B\Zp^4},
\\
\label{lemma:Thetheta:eq4}
\tT{p{\cal C}_{k+2}^0, P_1\Zp^4}^{\theta} &=& p^4\tep\tT{{\cal C}_{k+4}^0, P_1\Zp^4},\\
\tT{p^{k+2}}^\theta &=& p^2\tT{A^{k+2}, p\Zp^4} + p^3\tT{{\cal C}_{k+2}^0, p\Zp^4}\nonumber\\
\label{lemma:Thetheta:eq5}
&+& p^4\tT{p{\cal C}_{k}, p\Zp^4}.
\end{eqnarray}
\end{lemma}

\begin{proof}
They are immediate consequences of Proposition \ref{BasicTheory:DoubleCoset}, Proposition \ref{Thetheta:tT(A)theta}, 
and the facts that
$\gr^{A^l} \subset \gr^{A}$ for each $l \geq 1$ and that $\gr^{C} \subset \gr^{P_1}$
for each $C \in {\cal C}_{l}^0$ with $l \geq 2$.

\end{proof}

\begin{proposition}\label{prop:TATpnTheta}
For each nonnegative integer $k$, $\tT{A,\vzero} \tT{p^{k+1}}^\theta$ equals
\eqna{
\tT{{\cal C}_{k+2}^0, P_1\Zp^4}^{\theta} 
+p^4(1+p)\tep\tT{p^{k}}
+p^3\tep\tT{p^{k}}^\theta+
 \frac{1}{p^2}\tT{p^{k+2}}^{\theta^2}.
}
\end{proposition}



{
\newcommand{\tmp}{{\mathcal X}(\alpha,\beta)}
\newcommand{\tmpa}[1]{{\mathcal Y}\left(\alpha,\beta,#1\right)}
\newcommand{\tmpb}{{\mathcal N}}
\begin{proof}
By Lemma \ref{HeisenbergLiealgebras:lemma:tTAzerotTB}
\[\tT{A,\vzero}\tT{p^{k+1}}
= S_A(\{A^{k+2}\})
+ S_A({\cal C}_{k+2}^0)
+ S_A(p{\cal C}_{k})
.\]
Let us put 
\[\tmp = {\mathcal X}\left(A,C(\alpha,\beta,k+2)\right),\]
and 
\[\tmpa{\va} = {\mathcal Y}\left(A,C(\alpha,\beta,k+2), \va \right).\]	
We shall calculate $\left|\gr\bs\tmpa{\va}\right|$.
Notice that, for all $\vb \in \Z_p^4$, \[ \tmpa{\va}= \tmpa{\va + p\vb}.\]
The calculation will be divided into three cases.
\begin{enumerate}
\item If
	$\alpha = 0,\beta=0$, then Corollary \ref{Calculation:cor:ggAg} implies
	$\tmp 
		=
		\gr A_1.
	$
Hence it is easy to see that
	\[
	\left\{
	\mat{
		\left|\gr\bs\tmpa{\va}\right| = 1, & {\rm if} \ \va \in  A\Zp^4,
		\\
		\\
		\left|\gr\bs\tmpa{\va}\right| = 0, & {\rm otherwise}.
	}
	\right.
	\]

Hence it is easy to see that
	\eqna{
	S_A(\{A^{k+2}\})
	&=& \tT{A^{k+2}, A\Zp^4}.
}
By Proposition \ref{BasicTheory:DoubleCoset}, we have
\[\tT{A^{k+2}, A\Zp^4} 
= \tT{A^{k+2}, p\Zp^4}.
\]
Hence, we have
	\begin{equation}\label{equation:gAg1}
	S_A(\{A^{k+2}\})
	= \tT{A^{k+2}, p\Zp^4}.
	\end{equation}
\item If
	$\alpha = 0,\beta \geq 1$, then
	\[\tmp=\set{
            		\pmat{	p*	& * 	&* 	&* \\
            				p*	&*	&*	&* \\
            				p*	& * 	&* 	&* \\
            				p*	&*	&*	&* \\
            		}
		}
		\cap
		\gr A \gr,\]
	which is a right
	$\gr'	$
	invariant set. Hence  for all $Y \in \gr'$,
	\[\tmpa{\va}Y = \tmpa{Y^{-1} \va}.\]
	Thus, for all $Y \in \gr'$,
	\[\card{\gr \bs \tmpa{\va}} = \card{\gr \bs\tmpa{Y \va}}.\] 	
	By Corollary \ref{Caluculation:cor:Decomposition},
	we are reduced to the case where
	$
		\va \in \set{\ve_3, \ve_2,\ve_1,\vzero}.
	$
	By Corollary \ref{Calculation:cor:ggAg}, it is easy to see
	\[
	\begin{array}{ll}	
			\left| \gr \bs \tmpa{\ve_3}\right|  = 0,  
			&
			\left| \gr \bs \tmpa{\ve_2}\right|  = 1,\\
			\left| \gr \bs \tmpa{\ve_1}\right| = 1+p,
			&
			\left| \gr \bs \tmpa{\vzero}\right| = 1+p.
	\end{array}
	\]
	Thus, we have
	\eqna{
	S_A({\cal C}_{k+2}^0)
	&=&
	 \tT{{\cal C}_{k+2}^0, 
	 		\gr'
 \ve_2
            		\cup
            		\gr'
            		\ve_1
            		\cup
 			p\Zp^4
	 	} 
	+  p\tT{{\cal C}_{k+2}^0, 
			\gr'
            		\ve_1
            		\cup
			p\Zp^4
			}\\
	&=& 
	\tT{{\cal C}_{k+2}^0, P_1\Zp^4} 
	+  p\tT{{\cal C}_{k+2}^0, P_3\Zp^4}.
	}
By Proposition \ref{BasicTheory:DoubleCoset}, we have
\[\tT{{\cal C}_{k+2}^0, P_3\Zp^4} 
= \tT{{\cal C}_{k+2}^0, p\Zp^4}.\]
Hence, we have
	\begin{eqnarray}
S_A({\cal C}_{k+2}^0)	=
	\tT{{\cal C}_{k+2}^0, P_1\Zp^4} 
	+  p\tT{{\cal C}_{k+2}^0, p\Zp^4}.
	\end{eqnarray}
\item If 
	$\alpha \geq 1$, then
	$\tmp = \gr A \gr$	
	is naturally a right $\gr$ invariant set.
	Since 
	\[
	\mat{	
			\left| \gr \bs \tmpa{\ve_1}\right|  = 1+p,  
			&
			\left| \gr \bs \tmpa{\vzero}\right| = 1+p+p^2+p^3,  
	}
	\]
	we have
	\[
	S_A(p{\cal C}_{k})
		=
		(1+p)\tT{p{\cal C}_{k}, \Zp^4} + (p^3+p^2)\tT{p{\cal C}_{k}, p\Zp^4}.
	\]
By (\ref{lemma:Thetheta:eq1}) of Lemma \ref{lemma:Thetheta},
	\begin{eqnarray}
	\label{equation:gAg2}
	S_A(p{\cal C}_{k})^\theta
		&=& p^4(1+p)\tep\tT{p^k} + p^3\tep\tT{p^k}^\theta \nonumber \\
		&+& p^2\tT{p{\cal C}_{k}, p\Zp^4}^\theta.
	\end{eqnarray}
\end{enumerate}

Equality (\ref{lemma:Thetheta:eq5}) of Lemma \ref{lemma:Thetheta} and (\ref{equation:gAg1})$-$(\ref{equation:gAg2}) imply the desired identity.

\end{proof}
}

\begin{proposition}\label{prop:TBTpn}
For each nonnegative integer $k$,
\eqna{
\tT{B,\vzero}\tT{p^{k+2}}^{\theta^2} 
&=&
 \tT{{\cal C}_{k+4}^0, B\Zp^4}^{\theta^2}  + \tT{p{\cal C}_{k+2}^0, P_1\Zp^4}^{\theta^2} \\
&+& p^8 \tepp \tT{p^k} + p^7 \tepp \tT{p^k}^{\theta}\\
&+&(p^2 + p - 1) \tep \tT{p^{k+2}}^{\theta^2}.\\
&+& 	\tep \tT{p^{k+2}}^{\theta^3}\\ 
}
\end{proposition}
{
\newcommand{\tmp}{{\mathcal X}(\alpha,\beta)}
\newcommand{\tmpa}[1]{{\mathcal Y}\left(\alpha,\beta,#1\right)}
\newcommand{\tmpb}{{\mathcal N}}
\begin{proof}
By Lemma \ref{HeisenbergLiealgebras:lemma:tTAzerotTB}, we have
	\[\tT{B,\vzero}\tT{p^{k+2}}
	= S_B({\cal C}_{k+4}^0) 
	+ S_B(\{pA^{k+2}\})
	+S_B(p{\cal C}_{k+2}^0) 
	+S_B(p^2{\cal C}_{k}) .\]	
Note that $S_B(\{A^{k+4}\}) = 0$, by corollary \ref{Calculation:cor:ggBg}.
Let us put \[\tmp = {\mathcal X}\left(B,C(\alpha,\beta,k+4)\right),\]
and \[\tmpa{\va}= {\mathcal Y}\left(B,C(\alpha,\beta,k+4), \va \right).\]
We shall calculate $\left|\gr\bs\tmpa{\va}\right|$.
Notice that, for all $\vb \in \Z_p^4$,
\[ \tmpa{\va}= \tmpa{\va + p^2\vb}.\] 
The calculation will be divided into four cases.

\begin{enumerate}
\item
	If $\alpha = 0,\beta \geq 1$, then
	$\tmp =
		\gr B_1(0).
	$
	Hence it is easy to see that
	\begin{equation}\label{equation:gBg1}
		S_B({\cal C}_{k+4}^0)=
				 \tT{{\cal C}_{k+4}^0, B\Zp^4}.
	\end{equation}
\item
	If $\alpha = 1,\beta = 1$, then
	\[\tmp =
        		\set{
            		\pmat{	p*	& p* 	&* 	&* \\
            				p*	&p*	&*	&* \\
            				p*	& p* 	&* 	&* \\
            				p*	&p*	&*	&* \\
            		}
		}
		\cap
		\gr B \gr.	
	\]
	Hence
	$\tmp$ is right $\gr^A$ invariant.
	By Corollary \ref{Caluculation:cor:Decomposition}, we are reduced to the case where 
	$\va \in \set{\ve_3, \ve_1,
	 p\ve_3, p\ve_1,
	\vzero}$.
	By cor \ref{Calculation:cor:ggBg},  we have
	\eqna{
		\left| \gr \bs \tmpa{\ve_3} \right| &=& 0, \\ 
		\left| \gr \bs \tmpa{\ve_1} \right| &=& 1,  \\
		\left| \gr \bs \tmpa{p\ve_3} \right| &=& p,\\
		\left| \gr \bs \tmpa{p\ve_1} \right| &=& p+p^2,\\
		\left| \gr \bs \tmpa{\vzero} \right| &=& p+p^2.\\
	}
By Lemma \ref{Caluculation:lemma:Decomposition}, $S_B(\{pA^{k+2}\})$ equals
	\eqna{\tT{pA^{k+2}, A\Zp^4-p\Zp^4}
			+  p\tT{pA^{k+2}, p\Zp^4}
			+  p^2\tT{pA^{k+2}, pA\Zp^4}.
	}
By Lemma \ref{Calculation:lemma:DisjointSum}, we have
\[\tT{pA^{k+2}, A\Zp^4-p\Zp^4} = \tT{pA^{k+2}, A\Zp^4}-\tT{pA^{k+2},p\Zp^4}.\]
By (\ref{lemma:Thetheta:eq2}) of Lemma \ref{lemma:Thetheta}, 
we have
	\begin{eqnarray}
	S_B(\{pA^{k+2}\})^{\theta}
		&=& (p^2+p-1)\tT{pA^{k+2}, p\Zp^4}^{\theta}\nonumber \\
		&+&  p^2\tT{pA^{k+2}, p^2\Zp^4}^{\theta}.
	\end{eqnarray}
\item If
	$\alpha = 1, \beta \geq 2$, then
	\[\tmp =
		\set{
            		\pmat{	p*	& * 	&* 	&* \\
            				p*	&*	&*	&* \\
            				p*	& * 	&* 	&* \\
            				p*	&*	&*	&* \\
            		}
		}
		\cap
		\gr B \gr
	.\]
	Hence
	$\tmp$ is a right
	$
		\gr'
	$
	invariant set.
	By Corollary \ref{Caluculation:cor:Decomposition}, we are reduced to the case where 
	$\va \in \set{\ve_3, \ve_2,\ve_1,
	 p\ve_3, p\ve_2,p\ve_1,
	\vzero}$.
	By Lemma \ref{Calculation:cor:ggBg}, we have
	\eqna{
	\left| \gr \bs \tmpa{\ve_3} \right| &=& 0
	,\\  
	\left| \gr \bs \tmpa{\ve_2} \right| 
	&=& 1,\\
	\left| \gr \bs \tmpa{\ve_1} \right| &=& 1
	,\\ 
	\left| \gr \bs \tmpa{p\ve_3} \right| 
	&=& p + p^2,\\
	\left| \gr \bs \tmpa{p\ve_2} \right| &=& p + p^2
	,\\
	\left| \gr \bs \tmpa{p\ve_1} \right| 
	&=& p + p^2 + p^3 ,\\
	\left| \gr \bs \tmpa{\vzero} \right| &=& p + p^2 + p^3 
	.
	}
Thus  $S_B(p{\cal C}_{k+2}^0)$ equals
	\eqna{\tT{p{\cal C}_{k+2}^0, P_1\Zp^4-p\Z_p^4}
			+  (p+p^2)\tT{p{\cal C}_{k+2}^0, p\Zp^4}
			+  p^3\tT{p{\cal C}_{k+2}^0, pP_3\Zp^4}.
	}
Proposition \ref{BasicTheory:DoubleCoset} implies
$\tT{p{\cal C}_{k+2}^0, pP_3\Zp^4} = \tT{p{\cal C}_{k+2}^0, p^2\Zp^4}$.
Thus we have
\begin{eqnarray}
		 S_B(p{\cal C}_{k+2}^0)&=& \tT{p{\cal C}_{k+2}^0, P_1\Zp^4} \nonumber\\ 
		&+&  (p^2+p-1)\tT{p{\cal C}_{k+2}^0, p\Zp^4}\nonumber\\
			&+&  p^3\tT{p{\cal C}_{k+2}^0, p^2\Zp^4}.
\end{eqnarray}
\item If
	$\alpha  \geq 2$, then
	$\tmp 
		= \gr B \gr
	$
	is naturally right $\gr$ invariant set. Since
	\eqna{
	\left| \gr \bs \tmpa{\ve_1} \right| &=& 1
	,\\  
	\left| \gr \bs \tmpa{p\ve_1} \right| 
	&=& p+p^2+p^3
	,\\
	\left| \gr \bs \tmpa{\vzero} \right| 
	&=& p+p^2+p^3 + p^4,\\
	}
$S_B(p^2{\cal C}_{k})$ equals
	\eqna{\tT{p^2{\cal C}_{k}, \Zp^4}
			+  (p^3+p^2+p-1)\tT{p^2{\cal C}_{k}, p\Zp^4}+p^4\tT{p^2{\cal C}_{k}, p^2\Zp^4}.
	}	
By (\ref{lemma:Thetheta:eq1}) of Lemma \ref{lemma:Thetheta},
\[\tT{p^2{\cal C}_{k}, \Zp^4}^{\theta^2} = p^8 \tepp \tT{p^k},\ 
\tT{p^2{\cal C}_{k}, p\Zp^4}^\theta = p^4 \tepp \tT{p^k}.\]	
Thus we have
	\begin{eqnarray}
	\label{equation:gBg2}
		S_B(p^2{\cal C}_{k})^{\theta^2}
		  &=& p^8 \tepp \tT{p^k}+p^7 \tepp \tT{p^k}^{\theta}\nonumber\\
			&+&  (p^2+p-1)\tT{p^2{\cal C}_{k}, p\Zp^4}^{\theta^2}\nonumber\\
			&+&p^4\tT{p^2{\cal C}_{k}, p^2\Zp^4}^{\theta^2}.
	\end{eqnarray}	
\end{enumerate}
\bigskip	
	
Combining (\ref{lemma:Thetheta:eq5}) of Lemma \ref{lemma:Thetheta} and (\ref{equation:gBg1})--(\ref{equation:gBg2}), we complete the proof.
\end{proof}
}
\bigskip
By (\ref{lemma:Thetheta:eq3}) and (\ref{lemma:Thetheta:eq4}) of Lemma \ref{lemma:Thetheta},
we thus have the following identity:

%
\begin{corollary}\label{cor:maincor}
For all nonnegative integers $k$,
\eqna{
& & p^9\tep T(A)^{s}\tT{p^{k+1}}^\theta 
+ p^2T(A)^{s}\tT{p^{k+3}}^{\theta^3}
\\
&=&
 p^{14}\tepp\tT{p^{k}} 
 +p^5\left(T(B)^{s} + ({p^2}+1)\tep \right)\tT{p^{k+2}}^{\theta^2} 
 +\tT{p^{k+4}}^{\theta^4}.
}
\end{corollary}
\bigskip

The above corollary implies the main thoerem. 

\begin{theorem}\label{thm:main}
Problem $\ref{myprob}$ holds for $n=2$. Namely,
put
\eqna{
q_1 &=& -T(A),\\ 
q_2 &=& pT(B) + p(p^2+1)\ep,\\ 
q_3 &=& -p^3\ep T(A),\\ 
q_4 &=& p^6\epp,\\ 
g(X) &=& 1 - p^2 \ep X^2,
}
then we have

\begin{eqnarray}
\label{equation:main}
g^s(Y)
&=& \tP_2^{\theta^{4}}(X) 
+ q^s_1Y \tP_2^{\theta^{3}}(X) \nonumber \\
&+& q^s_2Y^2 \tP_2^{\theta^{2}}(X)
+ q^s_3Y^3 \tP_2^{\theta}(X)
+ q^s_4Y^4 \tP_2(X),
\end{eqnarray}
where $Y = p^2X$.
\end{theorem}
\medskip

\begin{remark}
Our theorem recovers Shimura's rationality via the morphism $\phi$.
\end{remark}
}

%

\textsc{Department of Health Informatics, Faculty of Health and Welfare Services Administration, Kawasaki University of Medical Welfare, Kurashiki, 701-0193, Japan}

{\it Email address}: \texttt{fumitake.hyodo@mw.kawasaki-m.ac.jp}

\end{document}